\RequirePackage{fix-cm}
\documentclass[smallextended]{svjour3}       
\smartqed  
\usepackage{hyperref}
\usepackage{mathrsfs}
\usepackage{amsmath}
\usepackage{amssymb}
\newcommand{\norm}[1]{\left\Vert#1\right\Vert}
\newcommand{\abs}[1]{\left\vert#1\right\vert}
\DeclareMathOperator*{\esssup}{ess\,sup}
\DeclareMathOperator*{\essinf}{ess\,inf }
\newcommand{\Dt}{\Delta}
\newcommand{\dv}{{\rm div}}

\newcommand{\rot}{\rm rot}
\newcommand{\real}{\mathbb{R}} 

\newcommand{\ds}{\displaystyle}      
\newcommand{\la}{\langle}
\newcommand{\ra}{\rangle}
\begin{document}

\title{Hemivariational Inequality for Navier-Stokes Equations: Existence,
Dependence and Optimal Control}


\titlerunning{Hemivariational Inequality for Navier-Stokes Equations}        

\author{H. Mahdioui         \and S. Ben Aadi \and K. Akhlil.
}


\institute{ \'{E}cole Nationale des Sciences Appliqu\'ees, Ibn Zohr. \\  Polydisciplinary faculty of Ouarzazate, Ibn Zohr}
           \date{Received: date / Accepted: date}

\maketitle

\begin{abstract}
In this paper we study existence, dependence and optimal control results concerning solutions to a class of hemivariational inequalities for stationary Navier-Stokes equations but without making use of the theory of pseudo-monotone operators. To do so, we consider a classical assumption, due to J. Rauch, which constrains  us to make a slight change on the definition of a solution. The Rauch assumption, although it insures the existence of a solution, does not allow the conclusion that the non-convex functional is locally Lipschitz. Moreover, two dependence results are proved, one with respect to changes of the boundary condition and the other with respect to the density of external forces. The later one will be used to prove the existence of an optimal control to the distributed parameter optimal control problem where the control is represented by the external forces.

\end{abstract}



\section{Introduction}
\label{intro}

This paper is devoted to the study of Navier-Stokes equations involving subdifferential boundary conditions but without making use of the theory of pseudo-monotone operators. We assume the nonslip boundary condition together with a Clarke subdifferential relation
between the pressure and the normal components of the velocity. Navier-Stokes equations together with this type of boundary condition model, in practice, the motion of an incompressible viscous fluid that, when pumped into the domain, can leave through the orifices on the boundary and, by a mechanism allowing the adjustment of the orifice's dimensions, the normal velocity on the boundary of the fluid is regulated to reduce the dynamic pressure.

Let $\mathscr O$ be a bounded simply connected domain in $\mathbb R^d$ with connected boundary $\partial\mathscr O$ of class $C^2$  ($d=2,\,3$). The stationary Navier-Stokes equations are described by the following system:
\begin{align}
 -\nu & \sum_{j=1}^d \frac{\partial^2 u_i}{\partial x_j^2} + \sum_{j=1}^du_j\frac{\partial u_i}{\partial x_j} + \frac{\partial p}{\partial x_i}  = f_i, \quad i=1,2,...,d \text{ in } \mathscr O , \label{p1}\\
 &\sum_{j=1}^d \frac{u_j}{x_j}=0 \text{ in } \mathscr O.\label{p2}
\end{align}where $u=\{u_i\}_{i=1}^d$ and $p$ are respectively the velocity and the pressure of the fluid. The external forces are represented by $f=\{f_i\}_{i=1}^d$ and the kinematic viscosity by the constant $\nu$. Using the standard Lamb formulation (\cite[Chapter I]{girault2012finite}), one can rewrite the equations \eqref{p1}-\eqref{p2} in an equivalent form involving the rotational operator and the dynamic pressure $\widetilde p:=p+\frac{1}{2}|u|^2$. The new formulation of the problem is then considered with the following boundary conditions:
\begin{equation}\label{boundary}
\widetilde p(z)\in\partial j(z,u_N(z)) \text{ and } u_\tau=0 \text{ on }\partial\mathscr O
\end{equation}
Here $u_N=u.n$ and $u_\tau=u-u_N\,n$ denote the normal and the tangential components of $u$ on the boundary $\partial\mathscr O$, $n$ being the unit outward normal vector on $\partial\mathscr O$. The multivalued mapping $\partial j$ denotes the Clarke subdifferential of a locally Lipschitz function $j(x,.)$.

In some important applications, but also in our present paper, the function $j$ can be expressed as 

\begin{equation}\label{j}
j(t)=\int_0^t\Theta(s)\,ds
\end{equation}for a locally bounded function $\Theta$ in $\mathbb R$ such that $\Theta(t\pm 0)$ exists fo all $t\in\mathbb R$. In this situation, we consider the following classical assumption introduced by J. Rauch \cite{rauch1977discontinuous} to study discontinous semilinear differential equations

\[
\displaystyle\esssup_{ ]-\infty,-t_{0}[ } \Theta(t) \leq 0 \leq \displaystyle\essinf _{ ]t_{0},+\infty[ } \Theta(t)
\]
This assumption will refer to us as \textit{Rauch assumption}. Geometrically, it describes the ultimate increase of the graph of the function $\Theta$. One of the most important advantage of this choice of $j$ is that it simplifies tremendously the calculation of the subdifferential $\partial j$. In fact, due to K. C. Chang \cite{chang1981variational}, the subdifferential of $j$ can be obtained by "filling in the gaps" in the discontinuous graph of $\Theta$. Under Rauch assumption, the resulting weak formulation of the problem \eqref{p1}-\eqref{boundary} is not a variational one but leads to the so-called \textit{hemivariational inequality}. When regularized with the help of Galerkin method the problem becomes a semilinear differential equation as discussed, in its simplest form, in the seminal work of J. Rauch \cite{rauch1977discontinuous}. This simple remark allows us to say that the hemivariational inequality can be seen, at least in our context, as a limit of a sequence of semilinear equations involving nonmonotone discontinuous functions.

It is fundamental to mention that, without any additional growth hypothesis on $j$, the Rauch assumption is sufficient to establish the existence of solutions to \eqref{p1}-\eqref{boundary}. Unfortunately, this condition does not make the functional  

\[
J(u)=\int_{\partial\mathscr O}j(u)\,d\sigma
\]locally Lipschitz or even finite on the whole space. Because of this reason, the Aubin-Clarke result giving the relation between the subdifferential of J and j can not be used. One strategy to encounter this problem is to modify slightly the definition of being a solution of \eqref{p1}-\eqref{boundary}.

It is worth to mention that the theory of hemivariational inequalities was introduced 
 the first time by Panagiotopoulos \cite{Pana81,Pana83,Pana89,Pana88,Panagiotopoulos1}, who by applying the generalized gradient of Clarke-Rockafellar \cite{Clar75,Clar83,Rock80} studied such variational-like expressions to discuss solutions of a class of  mechanical problems involving nonconvex and nonsmooth energy functionals. In the case of functions $j$ expressed as in \eqref{j}, hemivariational inequalities was extensively studied both in a mathematical and mechanical point of view, see  \cite{Pana89,Pana88,Pana91,Pana882,Mopana88,Panasta88,Panakol90,Panaban84} for more details.

The hemivariational inequalities for stationary and non-stationary Navier-Stokes equations were considered by many researchers in recent years. For convex functions $j(x,.)$, the problem has been studied by Chebotarev \cite{chebo92,chebo97,chebo03}. The boundary condition \eqref{boundary}, in the convex case, has been also considered for the Boussinesq equations  in \cite{chebo01} and in \cite{kono00} for its evolution counterpart. In all these papers the considered problems was formulated as variational inequalities. In the nonconvex case, the formulation of  \eqref{p1}-\eqref{boundary} is no longer a variational inequality but it leads to an hemivariational inequality. In the stationary case, the problem \eqref{p1}-\eqref{boundary} with nonconvex superpotentials $j$ was considered by Mig\'{o}rski and Ochal \cite{M05HE} Mig\'{o}rski \cite{mig04}, for non-Newtonian case see \cite{dudkalmig15}. In Orcliz spaces, hemivariational inequalities for Newtonian and Non-newtonian Navier-Stokes equations has been recently studied in \cite{migpac2018}, \cite{migpac19}. Hemivariational inequalities for generalized Newtonian fluids are recently extensively studied see \cite{dudkalmig17} and references therein, see also  \cite{migdud18} for evolutionary Oseen model for generalized Newtonian fluid. For an equilibrium problem approach to hemivariational inequalities for Navier-Stokes equations we refer to \cite{homancha19} and \cite{aadi18}. For different aspects about nonsmooth optimization in the context of Navier-Stokes system we refer to \cite{fang2019finite,fang16,kalita2015large,kalita2014attractors,M19F,M04he,MH,M19F,migorski2018evolutionary,SZ17Ev}.

The goal of this paper is threefold. We aim to
\begin{itemize}
 \item[(1)] show the existence of weak solutions to the hemivariational inequality corresponding to the problem \eqref{p1}-\eqref{boundary},
 \item[(2)] prove a dependence result of solutions with respect to the hemivariational part and to the density of the external forces,
 \item[(3)] formulate and study the distributed parameter optimal control where the control is represented by the density of the external forces.
\end{itemize} 

The paper is organized as follows. In section 2 we give the formulation of the stationary Navier-Stokes equations with a subdifferential boundary condition as an hemivariational inequality. We give a slight different definition to this problem to have a solution. This definition is so formulated to overcome the problem of the integrability and the local Lipschitzianity of $J$. We give in the end of this section an example illustrating the practicality of this model. In section 3, we first regularize the problem and then use Galerkin approximation. The existence of solutions to the regularized finite-dimentional problem is proven by using the Brouwer's fixed point theorem. In addition a weak precompactness result is obtained by the Dunford-Pettis theorem. In section 4, we prove the existence of a solution to our problem and we discuss why, in our opinion the question of uniqueness is difficult to answer even with a monotonicity assumption similar to the one in \cite{M05HE}. In section 5, we prove the dependence of the solution with respect to changes of the boundary condition by using an Aubin-Frankowski theorem. Finally in section 6, we first prove the dependence of solutions on external forces and use the result to prove the existence of an optimal control to a distributed parameter optimal control problem formulated by considering the external forces as controls.

\section{Problem statement}
\label{sec:2}
Let $\mathscr O$ a bounded simply connected domain in $\real^d$ with connected boundary $\partial \mathscr O$ of class $C^2$ ($d=2,3 $). We consider the following Navier-Stokes system:
\begin{equation}\label{e}
 -\nu \Dt u + (u.\nabla u) + \nabla p  = f, \quad \dv u = 0 \,
  \mbox{ in } \mathscr O.
\end{equation}
This system describes the flow of incompressible viscous fluid in the domain $\Omega$, subjected to the external forces $f=\{f_i\}_{i=1}^d$.  The unknown are the velocity $u =\{u_i\}_{i=1}^d$ and the pressure $p$ of the fluid . The positive constant $\nu$ is the kinematic viscosity of the fluid ($\nu  =\frac{1}{Re}$ where $Re$ stands for the Reynolds number). The nonlinear term $(u.\nabla)u$, called the convective term, is the symbolic notation of the vector $\sum_{j=1}^d u_j \frac{\partial u_i}{\partial x_j}$. The second condition, i.e $\dv\, u=0$, expresses the fact that the fluid is incompressible.

In order to give a variational-like formulation of (\ref{e}), we will use the approach developed by Chebotarev \cite{chebo92,chebo97,chebo01}, Konovalova \cite{kono00} and Alekseev and Smishliaev \cite{alekseev2001solvability}. By means of standard Lamb formulation \cite[Chapter I]{girault2012finite}, one obtain the following identities
\begin{equation}\label{lamb}
-\Dt u= \rot\,\rot\,\ u-\nabla\, \dv\,  u,  \end{equation}

\begin{equation}\label{lamb2}
(u.\nabla) u= \rot \,u \times u-\frac{1}{2} \nabla(u.u).
\end{equation}
Using the expressions \eqref{lamb}-\eqref{lamb2} and the incompressibility condition, the equation (\ref{e}) can be reformulated as follows
\begin{equation*}
 \nu \rot\,\rot\, u + \rot\, u\times u + \nabla \widetilde p  = f,
\end{equation*}
\begin{equation*}
 \quad \dv u = 0 \,
  \mbox{ in } \mathscr O .
\end{equation*}
where $\widetilde p=p+\frac{1}{2}\abs{u}^2$ is the total head of the fluid, or "total pressure" .

We suppose that, on the boundary $\partial\mathscr O$, the tangential components of the velocity vector are known and
without loss of generality we put them equal to zero (the nonslip condition):
\begin{equation}\label{eq2}
u_\tau:=u-u_N\,n=0 \text{ on } \partial\mathscr O,
\end{equation}
where $n$ is the unit outward normal on the boundary $\partial \mathscr O$ and $u_N=u.n$ denotes the normal component of the vector $u$. Moreover, we assume the following subdifferential boundary condition:
\begin{equation}\label{eq4}
\widetilde p(z) \in \partial j(z,u_N(z)) \text{ for } z\in \partial\mathscr O
\end{equation}where $\partial j (\xi)$ is the generalized gradient of   $j$ at $\xi$ and is given by 
\begin{center}
$\partial j(\xi)= \{ \xi^*\in V^* : j^0(\xi;h) \geq \la \xi^*,h\ra_{V^*\times V} \text{ for all } h\in V \}$,
\end{center}
$j^0(\xi;h)$ is the generalized derivative of a locally Lipschitz function $j$ at $\xi\in V$ in the direction $h\in V$ defined by:
\begin{center}
$j^0(\xi;h)=\ds\limsup_{\nu \rightarrow \xi, \lambda \downarrow
0} \frac{j(\nu+\lambda h)-j(\nu)}{\lambda }.$
\end{center}

In order to give the weak formulation of the problem (\ref{eq2})-(\ref{eq4}), we introduce the following functional spaces:
\begin{equation}
\mathcal{W}= \{ u \in \mathcal{C}^\infty(\mathscr O;\real^d) : \dv u=0 \text{ in } \mathscr O, u_\tau=0 \text{ on } \partial\mathscr O \}. 
\end{equation}
Let us denote by $\mathcal{V}$ and $\mathcal{H}$ the closure of $\mathcal{W}$ in the norms of $H^1(\mathscr O;\real^d)$ and $L^2(\mathscr O;\real^d)$, respectively. 
We define the operators $\mathscr A: \mathcal{V} \to \mathcal{V}^* $ and $\mathscr B[.]: \mathcal{V}\times \mathcal{V} \to \mathcal{V}^*$ with $\mathscr B[u]=\mathscr B(u,u)$ by
\begin{align*}
\la \mathscr A u,v\ra=&  \nu \int_{\mathscr O} \rot\, u.\rot\, v \,{\rm d}\lambda(x) \\ \la \mathscr B(u,v),w\ra= & \int_{\mathscr O} (\rot\, u\times v).w\,{\rm d}\lambda(x) 
\end{align*}
for $u, v, w \in \mathcal{V}$.

We multiply the equation of
motion (\ref{eq2}) by $v \in\mathcal{V}$ and apply the Gauss divergence theorem, we have:
\begin{equation}\label{eq6}
\la \mathscr Au+\mathscr B[u],v\ra+\ds\int_{\partial\mathscr O} \widetilde{p}(z)\,v_N(z)  \,d\sigma(z)=\la f,v\ra,
\end{equation}

From the relation (\ref{eq4}), by using the definition of the Clarke subdifferential, we have
\begin{equation}\label{eq7}
\ds\int_{\partial \mathscr O} \widetilde{p}(z)\, v_N(z) \,d\sigma(z) \leq \int_{\partial\mathscr O} j^0(z,u_N(z);v_N(z))\, d\sigma(z).
\end{equation}
The two relation (\ref{eq6})-(\ref{eq7}) yield to the following weak formulation 
\begin{equation*}
\text{(HVI)} \left\{ 
\begin{array}{ll}
\text{ Find } u\in \mathcal{V} \text{ such that }\\ \la \mathscr Au+\mathscr B[u],v\ra+\ds\int_{\partial\mathscr O} j^0(z,u_N(z);v_N(z))\, d\sigma(z) \geq \la f,v\ra, \text{ for every } v\in \mathcal{V},

\end{array}
\right.
\end{equation*}
the equation above is called an hemivariational inequality.

We have already mentioned in the introduction that the Rauch assumption is not sufficient to make the functional $J$ locally Lipchitz or even finite in the whole space $\mathcal V$. Because of this reason, a slight modified definition of being a solution should be adopted. Define the space $\widetilde{\mathcal V}$ as follows:
\[
\widetilde{\mathcal V}=\{u\in\mathcal V : \,\, u_N=\gamma(u).n\in L^\infty(\partial\mathscr O;\mathbb R)\}
\]where $\gamma$ is the trace operator from $\mathcal V$ in $L^2(\partial\mathscr O;\mathbb R^d)$. It is easy to prove that $\widetilde{\mathcal V}$ is dense in $\mathcal V$ for the weak topology. Now, we are able to give what we mean by a solution of the problem $(HVI)$.

\begin{definition}
A pair $(u,\xi)\in\mathcal V\times L^1(\partial\mathscr O,\,\mathbb R)$ is said to be solution of (HVI) if the following two relations are satisfied  

\begin{equation}\label{eqdef}
\left\{ 
\begin{array}{ll}
\la \mathscr Au+\mathscr B[u],v\ra+\ds\int_{\partial\mathscr O} \xi(z)\,v_N(z)\, d\sigma(z) = \la f,v\ra, \text{ for every } v\in \widetilde{\mathcal V}\\
~\\
\xi(z)\in\partial j(z;u_N(z)),\quad\text{for a.e. } z\in \partial \mathscr O.

\end{array}
\right.
\end{equation}

\end{definition}

Let us introduce the following operator $\mathscr E: L^1(\partial\mathscr O,\,\mathbb R)\rightarrow \widetilde{\mathcal V}^*$  defined by 
\[
\la \mathscr E(\xi),\,v\ra_{\mathcal V}=\int_{\partial\mathscr O} \xi(z)\,v_N(z)\, d\sigma(z),\quad \forall v\in \widetilde{\mathcal V}.
\]

In order to justify this definition, let us observe that for any $\xi\in L^1(\partial\mathscr O,\,\mathbb R^d)$ there may correspond no more that one linear continuous functional  $\mathscr E(\xi)\in\mathcal V^*$ with the property that 
 
\[
\la \mathscr E(\xi),\,v\ra_{\mathcal V}=\int_{\partial\mathscr O} \xi(z)\,v_N(z)\, d\sigma(z),\quad \forall v\in \widetilde{\mathcal V}
\]

This fact is a consequence of the density of the space $\widetilde{\mathcal V}$ in $\mathcal V$. Accordingly, if $\mathscr E(\xi)\in\mathcal V^*$, then for each $v\in \mathcal V$ the value $\la \mathscr E(\xi),\,v\ra_{\mathcal V}$ is determined uniquely by $\xi$. The equation \eqref{eqdef} can be written then as
\[
\la \mathscr Au+\mathscr B[u]-f,v\ra+\la \mathscr E(\xi),\,v\ra_{\mathcal V}=0,
\]or more compactly as 
\[
\Lambda(u,\xi)=f \text{ for all }v\in\mathcal V.
\]
For simplicity and if no ambiguity occurs, we write always our problem as
\[
\la \mathscr Au+\mathscr B[u],v\ra+\int_{\partial\mathscr O} \xi(z)\,v_N(z)\, d\sigma(z)=\la f,v\ra
\] for every $v\in\widetilde{\mathcal V}$.

The above procedure has been extensively used by Naniewicz \cite{nani94,nani95} to study hemivariational inequalities with directional growth conditions. Such non standard growth conditions give arise to problems involving functionals $J$ which are not locally Liptschitz in the whole space.

In the following remark we will highlight the fact that there is an equivalence, in some sense, between the Navier-Stokes system (\ref{eq2})-(\ref{eq4}) and the hemivariational inequality \eqref{eqdef}.
 
\begin{remark}
It's clear that the hemivariational inequality \eqref{eqdef} can be derived from (\ref{eq2})-(\ref{eq4}). Now we show that, in some sense, the converse also holds true. Let $(u,\xi)\in\mathcal V\times L^1(\partial\mathscr O;\mathbb R)$ be a solution to the problem \eqref{eqdef}, then by construction of $\mathcal{V}$, we have $\dv u=0$  and  $u_\tau=0$ on $\partial\mathscr O$. Now, let us take an arbitrary element $w$ in $\mathcal V\cap C^\infty_0(\mathscr O;\real^d)$, then also $w_N=0$ and one obtains that 
$\langle  \mathscr Au +\mathscr B[u], w\rangle = \langle f, w\rangle $. Note $\tilde f_u=f-\mathscr A u-\mathscr B[u]$, and by the density of $\mathcal V\cap C^\infty_0(\mathscr O;\real^d)$ in $\mathcal V$ we can write
\[
\la \tilde f_u,w\ra=0, \quad\text{for all }w\in \mathcal V 
\]
From Proposition 1.1 in Chapter
I of Temam \cite{Temam} it follows that there exist a distribution $h$ such that $\tilde f_u=\nabla h$. As a consequence we have
\[
\mathscr Au +\mathscr B[u]+ \nabla h = f
\] which, by multiplying by $v$ and integrating by
 parts over $\mathscr O$, implies 
\[
\langle \mathscr Au + \mathscr B[u], v\rangle + \int_{\partial\mathscr O} h(z)v_N(z) d\sigma(z) = \langle f, v\rangle
\]
Comparing this equality with the one in \eqref{eqdef} entails
\[
\int_{\partial\mathscr O}[h(z) - \xi(z)]v_N(z)] d\sigma(z)
= 0,\quad \,\forall \,v \in V.
\]
As $v$ is arbitrary, one can conclude that $h\in L^1(\partial\mathscr O,\mathbb R)$ and $h=\xi\in\partial j(z,u(z))$ a.e. This shows the subdifferential condition (\ref{eq4}).
\end{remark}

The following example shows the practicality of our framework
\begin{example}

The subdifferential condition appearing in the problem (\ref{eq2})-(\ref{eq4}) refers, in practice, to an artificial behaviour of the flow of the fluid through the boundary $\partial\mathscr O$. The fluid pumped into $\mathscr O$ can leave the domain through the orifices on the boundary. By a mechanism allowing the adjustment of the orifices dimensions, the normal velocity on the boundary of the fluid is regulated to reduce the dynamic pressure on $\partial\mathscr O$.

We consider the boundary condition (\ref{eq4}) by given real numbers $a$ and $b$ such that
 $ 0\leq a \leq b$. The locally Lipschitz function $j : \partial\mathscr O \times \real \rightarrow \real $
is defined by :
\begin{equation*}
 j(x,s) = \left\{
\begin{array}{lll}
\frac{\widetilde{p} }{2(b-a)} (s-a)^2  &\text{ if }  0 \leq s < b ; \\
 \frac{\widetilde{p}}{2  }(b-a) & \text{ if }  s \geq b.
\end{array}
\right.
\end{equation*}
For $x\in \partial\mathscr O$, we have:
\begin{equation*}
\partial j(x,s) = \left\{
\begin{array}{lll}
\frac{\widetilde{p} }{b-a} (s-a)  &\text{ if }  0 \leq s < b ; \\
\left[0,\widetilde{p}\right]  & \text{ if }  s= b ; \\
 0 & \text{ if }  s > b.
\end{array}
\right.
\end{equation*} 
The condition $u_N> 0$ refers to the fact that there is a flow through $\partial\mathscr O$. The boundary condition $u_N=0$ means that there is no flow across the boundary. If $u_N\in (0, b)$, the orifices allow the fluid to infiltrate outside the tube. When the velocity of the fluid increases, the total pressure is a linear function which takes its values between $ 0 $ and $ \widetilde{p} $. If $ u_N $ reaches the value $ b $, a mechanism opens the holes more widely and allows the fluid to pass to the outside. As a result, the pressure drops to $ 0 $. Finally, its worth to mention that the dependence of $j$ on the space variable traduces the fact that the subdifferential boundary condition can possibly take different values on the parts of $ \partial\mathscr O $. For other examples, see \cite[ Example 18]{M05HE} and \cite[Remark 2]{M07HE}.
\end{example}
\section{Regularized Problem}

In the forthcoming study of the problem \eqref{eqdef} we restrict ourselves to superpotentials $j$ which are independent of $z$.

Let $\Theta \in L^\infty_{loc}(\real)$. For $\mu>0$ and $t\in \real$, we define:
\begin{equation*}
\underline{\Theta}_\mu(t)=\displaystyle\essinf_{\abs{t-s}\leq\mu}\Theta(s),\quad  \quad \overline{\Theta}_\mu(t)=\displaystyle\esssup_{\abs{t-s}\leq\mu}\Theta(s).
\end{equation*}
For a fixed $t\in \real$, the functions $\underline{\Theta}_\mu$ and $\overline{\Theta}_\mu$ are decreasing and increasing in $\mu$, respectively. Let 
\begin{align*}
\underline{\Theta}(t)=\lim_{\mu \to 0^+} \underline{\Theta}_\mu(t), \quad\quad \overline{\Theta}(t)=\lim_{\mu \to 0^+}\overline{\Theta}_\mu(t),
\end{align*}
and let $\widehat{\Theta}(t):\real \to 2^\real$ be a multifunction defined by
\begin{equation*}
\widehat{\Theta}(t)=\left[ \underline{\Theta}(t),\overline{\Theta}(t) \right]
\end{equation*}
From Chang \cite{chang1981variational} we know that a locally Lipschitz
function $j:\real \to \real$ can be determined up to an additive constant by the relation $$j(t)=\int_0^t \Theta(s)\,d \lambda(s) $$ such that $\partial j(t) \subset \hat{\Theta}(t)$ for all $t\in\real$. If moreover, the limits $\Theta(t\pm 0)$ exist for every $t\in\real$, then $\partial j(t) = \widehat{\Theta}(t)$.
\begin{remark}
Here for the existence theory an abstract regularization procedure by convolution is used. Such a regularization procedure can be modified in order to get approximations of locally Lipschitz function that can be treated numerically; see \cite{ovg14}.
\end{remark}
Now, we consider the mollifier 
\[\mathfrak h\in C_0^\infty(-1,1), \mathfrak h\geq 0 \text{ with } \ds\int_{-\infty}^{+\infty} \mathfrak h(s)\, {\rm d}\lambda(s)=1
\] and let

\[\Theta_\varepsilon =\mathfrak h_\varepsilon * \Theta
\text{ with } \mathfrak h_\varepsilon(s)=\frac{1}{\varepsilon}\mathfrak h(\frac{s}{\varepsilon})
\]where $*$ denotes the  convolution product and $0<\varepsilon<\varepsilon_0$. Thus the regularized problem becomes:
\begin{equation*}
(\mathscr P_\varepsilon) \left\{ 
\begin{array}{ll}
\text{ Find } u \in \mathcal V \text{ such that:  for all } v\in \widetilde{\mathcal V}\\
\ds \la \mathscr Au+\mathscr B[u],v\ra+\int_{\partial{\mathscr O}} \Theta_\varepsilon(u_N)\,v_N\, d\sigma =\la f,v\ra. \end{array}
\right.
\end{equation*}

Now and in order to define the corresponding finite dimensional problem $(\mathscr P_\varepsilon)$, we consider a Galerkin basis of $\widetilde{\mathcal V}$ and let $\mathcal V_n$ be the resulting n-dimensional subspace. 
This problem reads: \\
\begin{equation*}
(\mathscr P^n_\varepsilon) \left\{ 
\begin{array}{ll}
\text{ Find } u^{\varepsilon_n} \in \mathcal V_n \text{ such that:  for all } v\in \mathcal V_n\\
\ds \la \mathscr Au^{\varepsilon_n}+\mathscr B[u^{\varepsilon_n}],v\ra+\int_{\partial{\mathscr O}} \Theta_{\varepsilon_n}(u_N^{\varepsilon_n})\,v_N d\sigma =\la f,v\ra .
\end{array}
\right.
\end{equation*}

For the existence of solutions we will need the following hypothesis $H(\Theta)$:

\begin{itemize}
\item[(1) ](Chang assumption) $\Theta \in L_{loc}^{\infty}(\real),\, \Theta(t \pm 0)\text{ exists for any } t \in \real.$
\item[(2) ] (Rauch assumption) there is  $t_{0} \in
\real$ such that:
\[
\displaystyle\esssup_{ ]-\infty,-t_{0}[ } \Theta(t) \leq 0 \leq \displaystyle\essinf _{ ]t_{0},+\infty[ } \Theta(t)
\]
\end{itemize}

\begin{remark}
If one assume more generally that

\begin{equation}\label{eqdefalpha}
\displaystyle\esssup_{ ]-\infty,-t_{0}[ } \Theta(t) \leq \alpha \leq \displaystyle\essinf _{ ]t_{0},+\infty[ } \Theta(t)
\end{equation}
for some real number $\alpha$, it is possible to come back to the situation where the Rauch assumption is imposed by simply replacing $\Theta$ by $\Theta -\alpha$ and $f$ by $f-\alpha$. In fact if we assume \eqref{eqdefalpha} the problem \eqref{eqdef} is equivalent to 

\begin{equation*}
\left\{ 
\begin{array}{ll}
\la \mathscr Au+\mathscr B[u],v\ra+\ds\int_{\partial\mathscr O} (\xi(z)-\alpha)\,v_N(z)\, {\rm d}\sigma(z) = \la f-\alpha,v\ra, \text{ for every } v\in \widetilde{\mathcal V}\\
~\\
\xi(z)-\alpha\in\partial j(z;u_N(z))-\alpha=\partial \left[j(z;u_N(z))-\alpha u_N(z)\right],\quad\text{for a.e. }z\in\partial\mathscr O

\end{array}
\right.
\end{equation*}
Let us note $\tilde f=f-\alpha$ and $\widetilde\Theta=\Theta-\alpha$. Thus, $\tilde j(.)=\int_0^{.}\widetilde\Theta(s)\,{\rm d}\lambda(s)$ and the problem \eqref{eqdef} under \eqref{eqdefalpha} become to find $(u,\eta)$ such that 

\begin{equation*}
\left\{ 
\begin{array}{ll}
\la \mathscr Au+\mathscr B[u],v\ra+\ds\int_{\partial\mathscr O} \eta\,v_N(z)\, d\sigma(z) = \la \tilde f,v\ra, \text{ for every } v\in \widetilde{\mathcal V}\\
~\\
\eta\in\partial \tilde j(z;u_N(z)),\quad\text{for a.e. }z\in\partial\mathscr O

\end{array}
\right.
\end{equation*}where $\widetilde\Theta$ fulfill the Rauch assumption. This means that, without loss of generality, we can always consider the initial Rauch assumption (with $\alpha=0$).
\end{remark}

\begin{lemma}\label{lem1}
Suppose that $H(\Theta)$ holds. Then we can determine $a,b > 0$ such that for every $u\in\mathcal V$
\begin{equation}\label{eq14}
\int_{\partial{\mathscr O}} \Theta_{\varepsilon}(u_N(z))u_N(z) \,{\rm d}\sigma(z) \geq -ab \,\sigma(\partial{\mathscr O}).
\end{equation}
\end{lemma}

\begin{proof}
From the hypothesis $H(\Theta)$ we obtain that
\begin{equation*}
\Theta_\varepsilon (\xi)=( \mathfrak h \star \Theta)(\xi)=\int_{-\varepsilon}^{+\varepsilon}\Theta(\xi-t) \mathfrak h_\varepsilon(t){\rm d}\lambda(t) \leq \displaystyle\esssup_{\abs{t}\leq \varepsilon} \Theta(\xi-t)\end{equation*} and analogously
\begin{equation*}\displaystyle\essinf_{\abs{t}\leq \varepsilon} \Theta(\xi-t) \leq \Theta_\varepsilon(\xi)
\end{equation*}
In the above two inequalities we set $x = \xi-t, \abs{x-\xi}\leq \varepsilon$ and enlarge the bounds for $-\infty <x\leq \varepsilon +\xi$ and $\xi-\varepsilon \leq x< \infty$, respectively. Then the supremum and
the infimum for $\xi \in (-\infty,-\xi_1)$ and $\xi\in (+\xi_1,+\infty)$, respectively are formed and the bounds are enlarged by replacing $\varepsilon +\xi$ by $1 -\xi_1$ and $\xi-\varepsilon$ by $\xi_1 - 1 (\varepsilon < 1)$;
we obtain from $\text{H}(\Theta)$ that there exists $\xi\in \real$ such that
\begin{equation*}
\sup_{(-\infty,-\xi)} \Theta_\varepsilon(\xi_1) \leq 0\leq \inf_{(+\xi,+\infty)} \Theta_\varepsilon(\xi_1).
\end{equation*}
Thus we can determine $a> 0$ and $b > 0$ such that

\begin{equation*}
\left\{ 
\begin{array}{ll}
\Theta_\varepsilon(\xi)\geq 0, & \text{ if }\xi >a, \\
\Theta_\varepsilon(\xi)\leq 0, & \text{ if }\xi <-a,\\
\abs{\Theta_\varepsilon(\xi)}\leq b, & \text{ if }\abs{\xi}\leq a,
\end{array}
\right.
\end{equation*}and may write
\begin{align*}
\ds\int_{\partial{\mathscr O}} \Theta_{\varepsilon_n}(u_N(z))u_N(z)\, {\rm d}\sigma(z)=&\ds\int_{\abs{u_N(z)}>a} \Theta_{\varepsilon_n}(u_N(z))u_N(z)\,{\rm d}\,\sigma(z)\\
&\qquad+\int_{\abs{u_N(z)}\leq a} \Theta_{\varepsilon_n}(u_N(z))u_N(z)\,{\rm d}\,\sigma \\ 
&\geq \int_{\abs{u_N(z)}\leq a} \Theta_{\varepsilon_n}(u_N(z))u_N(z)\,{\rm d}\,\sigma\\
& \geq - ab\, \sigma(\partial{\mathscr O}).
\end{align*}
\end{proof}
\begin{lemma}\label{lem.est.theta}
Suppose $\Theta\in L^\infty_{\mathrm{loc}}(\mathbb R)$. Then for every $t,t'\in\mathbb R$ we have
\begin{equation*}
|\Theta_\varepsilon(t)-\Theta_\varepsilon(t')|
\leq \esssup_{s\in]t\wedge t'-\varepsilon,t\vee t'+\varepsilon[}|\Theta(s)|\,\|\mathfrak h'\|_{\infty}\,|t-t'|\,\left(t\vee t'-t\wedge t'+2\varepsilon\right)
\end{equation*}
where $\|\mathfrak h'\|_{\infty}:=\|\mathfrak h'\|_{C_c^\infty(]-\varepsilon,\varepsilon[)}$.

\end{lemma}

\begin{proof}Let $t,\,t'\in\mathbb R$, we have the following estimates

\begin{align*}
|\Theta_\varepsilon(t)-\Theta_\varepsilon(t')|\leq&\ds\int_{\mathbb R}|\Theta(s)|\,|\mathfrak{h}_{\varepsilon}(t-s)-\mathfrak h_\varepsilon(t'-s)| {\rm d}\lambda(s)\\
& \leq \int_{t\wedge t'-\varepsilon}^{t\vee t'+\varepsilon}|\Theta(s)|\,|\mathfrak{h}_{\varepsilon}(t-s)-\mathfrak h_\varepsilon(t'-s)|\,{\rm d}\lambda(s) \\
&\leq\esssup_{s\in]t\wedge t'-\varepsilon,t\vee t'+\varepsilon[}|\Theta(s)|\,\int_{t\wedge t'-\varepsilon}^{t\vee t'+\varepsilon}|\mathfrak{h}_{\varepsilon}(t-s)-\mathfrak h_\varepsilon(t'-s)|\,{\rm d}\lambda(s).
\end{align*}

By mean value theorem, there exists $c\in ]t\wedge t'-\varepsilon,t\vee t'+\varepsilon[$ such that 
\[
\mathfrak{h}_{\varepsilon}(t-s)-\mathfrak h_\varepsilon(t'-s)=\mathfrak h'_\varepsilon(c)(t-t')
\]which give
\begin{align*}
|\Theta_\varepsilon(t)-\Theta_\varepsilon(t')|
&\leq\esssup_{s\in]t\wedge t'-\varepsilon,t\vee t'+\varepsilon[}|\Theta(s)|\,\int_{t\wedge t'-\varepsilon}^{t\vee t'+\varepsilon}|\mathfrak h'_\varepsilon(c)|\,|t-t'|\,{\rm d}\lambda(s)\\
&\leq \esssup_{s\in]t\wedge t'-\varepsilon,t\vee t'+\varepsilon[}|\Theta(s)|\,\|\mathfrak h'_\varepsilon\|_{\infty}\,|t-t'|\,\left(t\vee t'-t\wedge t'+2\varepsilon\right)
\end{align*}

\end{proof}

\begin{lemma}\label{lemK}
The operator  $\mathscr K_\varepsilon:\mathcal V_n\rightarrow\mathcal V_n^\star$ given by
\[
\la\mathscr K_{\varepsilon}(v),w\ra=\int_{\partial\mathscr O}\Theta_\varepsilon(v_N)\,w_N\,{\rm d}\,\sigma,\quad u,\,w\in \mathcal V_n
\] is weakly continuous.
\end{lemma}
\begin{proof}
Let $(u^k)_k$ be a sequence converging weakly to $u$ in $\mathcal V$. By Rellich's compactness criterion we may pass to a subsequence, which we still denote as $u^k$, so that $u_N^k=u^k.n$ converges to $u_N=u.n$ in $L^2(\partial\mathscr O)$ and then almost everywhere on $\partial\mathscr O$. It then follows that also $u_N^k\vee u_N$ and $u_N^k\wedge u_N$ converge to $u_N$. Thus by applying Egoroff's theorem we can find that for any $\alpha>0$ we can determine $\Gamma\subset \partial\mathscr O$ with $\sigma(\Gamma)<\alpha$ such that $u^k_N$, $u_N^k\vee u_N$ and $u_N^k\wedge u_N$ converge to $u_N$ uniformly on $\partial\mathscr O\setminus\Gamma$.

Let $\mu>0$, there exists $n_0$ such that for all $k\geq n_0$
\[
|u_N^k\vee u_N-u_N|<\frac{\mu}{2},
\]and 
\[
|u_N^k\wedge u_N-u_N|<\frac{\mu}{2}.
\]

If $s\in\mathbb R$ satisfies $u_N^k\wedge u_N-\varepsilon<s<u_N^k\vee u_N+\varepsilon$, then it satisfies $-\frac{\mu}{2}+u_N-\varepsilon<s<\frac{\mu}{2}+u_N+\varepsilon$. We choose $\mu$ such that $\frac{\mu}{2}>\varepsilon_0>\varepsilon$. Then $\varepsilon+\frac{\mu}{2}<\mu$ and $s$ satisfies $|s-u_N|<\mu$ which implies that $|s-\|u_N\||<\mu$ with $\|u_N\|=\|u_N\|_{L^\infty(\partial\mathscr O\setminus\Gamma)}$

By Lemma \ref{lem.est.theta}, we have for every $z\in\partial\mathscr O\setminus\Gamma$:
\begin{align*}
|\Theta_\varepsilon(u_N^k)-\Theta_\varepsilon(u_N)|
&\leq \esssup_{s\in]u_N^k\wedge u_N-\varepsilon,u_N^k\vee u_N+\varepsilon[}|\Theta(s)|\,\|\mathfrak h'\|_{\infty}\,|u_N^k-u_N|\,\\
&\qquad\qquad\times\left(u_N^k\vee u_N-u_N^k\wedge u_N+2\varepsilon\right)\\
&\leq \left(\mu+2\varepsilon\right)\esssup_{|s-\|u_N\||<\mu}|\Theta(s)|\,\|\mathfrak h'\|_{\infty}\,|u_N^k-u_N|\\
&\leq 2\mu\|\Theta\|_{L^\infty(]\|u_N\|-\mu,\|u_N\|+\mu[)}\,\|\mathfrak h'\|_{\infty}\,|u_N^k-u_N|\\
&\leq K_0\,|u_N^k-u_N|
\end{align*}
where $K_0=4k_0\varepsilon_0\|\Theta\|_{L^\infty(]\|u_N\|-2k_0\varepsilon_0,\|u_N\|+2k_0\varepsilon_0[)}\,\|\mathfrak h'\|_{C_c^\infty(]-\varepsilon_0,\varepsilon_0[)}$ for some $k_0\in ]0,1[$.

It follows that for all $v\in \mathcal V_n$
\begin{align*}
\int_{\partial\mathscr O\setminus\Gamma}\left|\Theta_\varepsilon(u_N^k)-\Theta_\varepsilon(u_N)\right|\,|v_N|\,{\rm d}\sigma(z)&\leq K_0\,\int_{\partial\mathscr O\setminus\Gamma} \left|u_N^k-u_N\right|\,|v_N|\,{\rm d}\sigma(z)\\
&\leq K_0\,\sigma(\partial\mathscr O\setminus \Gamma)\|v_N\|_{L^\infty(\partial\mathscr O\setminus\Gamma)}\|u_N^k-u_N\|_{L^\infty(\partial\mathscr O\setminus\Gamma)}
\end{align*}
As $\alpha$ is arbitrary we then conclude 

\[
\int_{\partial\mathscr O}\left|\Theta_\varepsilon(u_N^k)-\Theta_\varepsilon(u_N)\right|\,|v_N|\,{\rm d}\sigma(z)\leq K_1 \,\|u_N^k-u_N\|_{L^\infty(\partial\mathscr O)}\|v_N\|_{L^\infty(\partial\mathscr O)}
\]where $K_1:=K_0\,\sigma(\partial\mathscr O)$. It then follows that

\begin{equation*}
\left|\la\mathscr K_{\varepsilon}(u^k)-\mathscr K_{\varepsilon}(u),v\ra\right|\leq K_1 \,\|u_N^k-u_N\|_{L^\infty(\partial\mathscr O)}\|v_N\|_{L^\infty(\partial\mathscr O)}
\end{equation*}
Which complete the proof.
\end{proof}

\begin{proposition}\label{pro1}
Suppose that $H(\Theta)$ is satisfied. Then the regularized problem $(\mathscr P^n_\varepsilon)$  has at least one solution $u^{\varepsilon_n} \in \mathcal V_n$. Moreover the sequence $(u^{\varepsilon_n})_{n}$ is uniformly bounded on $\mathcal V$.
\end{proposition}

\begin{proof}
Let $\mathfrak i_n:\mathcal V_n\rightarrow \widetilde{\mathcal V}$ be the inclusion mapping of $\mathcal V_n$ into $\widetilde{\mathcal V}$ and $\mathfrak i_n^\star$ the dual projection mapping of $\widetilde{\mathcal V}^\star$ into $\mathcal V_n^\star$. Define $\mathscr A_n=\mathfrak i_n^\star\mathscr A\, \mathfrak i_n$, $\mathscr B_n[.]=\mathfrak i_n^\star\mathscr B[.]\, \mathfrak i_n$ and $f_n=\mathfrak i_n^\star f\in\mathcal V_n^\star$.

The regularized problem $(\mathscr P^n_\varepsilon)$ can be written equivalently in the form
\begin{equation}\label{e1}
\Lambda_{\varepsilon}(u^{\varepsilon_n})=0
\end{equation}
where $\Lambda_{\varepsilon}=\mathscr A_n+\mathscr B_n[.]+\mathscr K_{\varepsilon} - f_n$ from $\mathcal V_n$ into $\mathcal V_n^\star$. As the domain $\mathscr O$ is simply connected, it follows from \cite{bykhovskii1960orthogonal}, that the bilinear form 
\[\la\la u,v\ra\ra_{\mathcal V}= \ds\int_\Omega \rot\, u.\rot\, v \,d\lambda
\] 
generates a norm in $\mathcal{V}$ which is equivalent to the $H^1(\Omega;\real^d)$-norm. From this and from the Cauchy-Schwartz inequality one can deduce that there exist some $c>0$ such that 
\[
\la\mathscr A u,v\ra\leq c\,\|u\|_{\mathcal V}\|v\|_{\mathcal V},\quad\text{for all }u,\,v\in\mathcal V,
\]which means that $\mathscr A$ is continuous. Moreover, from \cite[Chapter II]{Temam} $\mathscr B[.]$ is continuous. This obviously implies that $\mathscr A_n$ and $\mathscr B_n[.]$ and consequently $\mathscr A_n+\mathscr B_n[.]$ are continuous. Finally by using Lemma \ref{lemK},  the continuity of $\Lambda_\varepsilon$ follows. Because of the coercivity of $\mathscr N:=\mathscr A+\mathscr B[.]$ (see \cite{aadi18}) and Lemma \ref{lem1} we have the estimate
\begin{equation}
\la \Lambda(u^{\varepsilon_n}),u^{\varepsilon_n}\ra \geq M\, \norm{u^{\varepsilon_n}}^2-ab \sigma(\partial{\mathscr O})-\|f\|\norm{u^{\varepsilon_n}},
\end{equation}where $M$ is the coerciveness constant of $\mathscr N$. By applying Brouwer's fixed point theorem (cf. \cite{lions1969quelques} p.53) we obtain that (\ref{e1})
admits a bounded solution $u^{\varepsilon_n}$. 

\end{proof}
\begin{proposition}\label{pro2}
The sequence $(\Theta_{\varepsilon_n}(u^{\varepsilon_n}))_n$ is weakly precompact in $L^1(\partial{\mathscr O})$.
\end{proposition}
\begin{proof}
The Dunford-Pettis theorem (cf. \cite{ekeland1976convex}, p.239) implies that it suffices to show that for each $ \mu>0$ a $\delta(\mu)>0$ can be determined such that for $\Gamma \subset \partial{\mathscr O}$ with $\sigma(\Gamma) < \delta$
\begin{equation}\label{eq101} 
\int_\Gamma \abs{\Theta_{\varepsilon_n}(u_N^{\varepsilon_n})}{\rm d}\sigma<\mu
\end{equation} 
The inequality
\[ s_0\abs{\Theta_\varepsilon(s)}\leq\abs{\Theta_\varepsilon(s)s}+s_0 \sup_{\abs{s}<s_0}\abs{\Theta_\varepsilon(s)}\]
implies that
\begin{equation}\label{eq100} \int_\Gamma \abs{\Theta_{\varepsilon_n}(u_N^{\varepsilon_n})}\,d\,\sigma \leq \frac{1}{s_0}  \int_{\partial\mathscr O} \abs{\Theta_{\varepsilon_n}(u_N^{\varepsilon_n})u_N^{\varepsilon_n}}\,{\rm d}\,\sigma+\int_\Gamma\displaystyle\sup_{\abs{u_N^{\varepsilon_n}(x)}\leq s_0}\abs{\Theta_{\varepsilon_n}(u_N^{\varepsilon_n})}\,{\rm d}\,\sigma
\end{equation}
But
\begin{align*} \int_{\partial\mathscr O} \abs{\Theta_{\varepsilon_n}(u_N^{\varepsilon_n})u_N^{\varepsilon_n}}\,{\rm d}\,\sigma=& \int_{\abs{u_N^{\varepsilon_n}}>\rho_1} \abs{\Theta_{\varepsilon_n}(u_N^{\varepsilon_n})u_N^{\varepsilon_n}}\, {\rm d}\,\sigma+\int_{\abs{u_N^{\varepsilon_n}}\leq\rho_1} \abs{\Theta_{\varepsilon_n}(u_N^{\varepsilon_n})u_N^{\varepsilon_n}}\,{\rm d}\,\sigma \\
=& \int_{\abs{u_N^{\varepsilon_n}}>\rho_1} \abs{\Theta_{\varepsilon_n}(u_N^{\varepsilon_n})u_N^{\varepsilon_n}}\,{\rm d}\,\sigma- \int_{\abs{u_N^{\varepsilon_n}}\leq\rho_1} \abs{\Theta_{\varepsilon_n}(u_N^{\varepsilon_n})u_N^{\varepsilon_n}}\, {\rm d}\,\sigma\\
&\qquad +2\int_{\abs{u_N^{\varepsilon_n}}\leq\rho_1} \abs{\Theta_{\varepsilon_n}(u_N^{\varepsilon_n})u_N^{\varepsilon_n}}\,{\rm d}\,\sigma \\
\leq & \int_{\abs{u_N^{\varepsilon_n}}>\rho_1} \abs{\Theta_{\varepsilon_n}(u_N^{\varepsilon_n})u_N^{\varepsilon_n}}\,{\rm d}\,\sigma+ \int_{\abs{u_N^{\varepsilon_n}}\leq\rho_1} \Theta_{\varepsilon_n}(u_N^{\varepsilon_n})u_N^{\varepsilon_n}\,{\rm d}\,\sigma\\
&\qquad +2\int_{\abs{u_N^{\varepsilon_n}}\leq\rho_1} \abs{\Theta_\varepsilon(u_N^{\varepsilon_n})u_N^{\varepsilon_n}}\,{\rm d}\,\sigma\\
= & \int_{\partial\mathscr O} \Theta_{\varepsilon_n}(u_N^{\varepsilon_n})u_N^{\varepsilon_n}\, {\rm d}\,\sigma+2\int_{\abs{u_N^{\varepsilon_n}}\leq\rho_1} \abs{\Theta_{\varepsilon_n}(u_N^{\varepsilon_n})u_N^{\varepsilon_n}}\,{\rm d}\,\sigma\\
=& \la f,u^{\varepsilon_n}\ra- \la \mathscr Au^{\varepsilon_n},u^{\varepsilon_n} \ra-\la \mathscr B[u^{\varepsilon_n}],u^{\varepsilon_n} \ra\\ 
&\qquad+2\int_{\abs{u_N^{\varepsilon_n}}\leq\rho_1} \abs{\Theta_{\varepsilon_n}(u_N^{\varepsilon_n})u_N^{\varepsilon_n}}\,{\rm d}\,\sigma \\
=& \la f,u^{\varepsilon_n}\ra- \la \mathscr Au^{\varepsilon_n},u^{\varepsilon_n} \ra +2\int_{\abs{u_{\varepsilon_n}}\leq\rho_1} \abs{\Theta_\varepsilon(u_N^{\varepsilon_n})u_N^{\varepsilon_n}}\,{\rm d}\,\sigma \\
\leq & c+2ab\,\sigma(\partial{\mathscr O}), \quad \text{ for some constant }c.
\end{align*}
In the last two inequalities we have used the boundedness of the solutions $(u^{\varepsilon_n})_n$, the estimate \eqref{eq14} and the relation
\begin{equation}\label{eq8}
\sup_{\abs{s}\leq s_0} \abs{\Theta_\varepsilon(s)}\leq \displaystyle\esssup_{\abs{s}\leq s_0+1} \abs{\Theta (s)},
\end{equation}
Now choose $s_0$ such that
for all $\varepsilon$ and $n$
\begin{equation}\label{eq9}
\frac{1}{s_0} \int_\Gamma\abs{\Theta_{\varepsilon_n}(u_N^{\varepsilon_n})u_N^{\varepsilon_n}}{\rm d}\sigma \leq \frac{1}{s_0}(c+2\rho_1\rho_2\sigma(\partial{\mathscr O})) \leq \frac{\mu}{2}
\end{equation}
and $\delta$ such that 
\begin{equation}\label{eq10}
\displaystyle\esssup_{\abs{s}\leq s_0+1} \abs{\Theta (s)} \leq \frac{\mu}{2\delta}
\end{equation}
Relation (\ref{eq8}) implies with (\ref{eq9}) that for $\sigma(\Gamma)<\delta$
\begin{equation*}
\int_\Gamma \sup_{\abs{u_N^{\varepsilon_n}}\leq s_0} \abs{\Theta_{\varepsilon_n}(u_N^{\varepsilon_n})}\, {\rm d}\, \sigma \leq \displaystyle\esssup_{\abs{u_N^{\varepsilon_n}}\leq s_0+1} \abs{\Theta_{\varepsilon_n}(u_N^{\varepsilon_n})} \sigma(\Gamma)\leq \frac{\mu}{2\delta}.\delta \leq \frac{\mu}{2}.
\end{equation*}
From the relations (\ref{eq100}), (\ref{eq9}) and (\ref{eq10}), the relation (\ref{eq101}) results,
i.e. that $\{\Theta_{\varepsilon_n}(u_N^{\varepsilon_n})\}$ is weakly precompact in $L^1(\partial{\mathscr O})$.

\end{proof}


\section{Existence}
In this section we present an existence result corresponding to the hemivariational inequality for Navier-Stokes systems under Rauch-Chang assumption $H(\Theta)$. The uniqueness question is discussed in Remark \ref{rem}.

\begin{theorem}\label{existence}
Under assumption $H(\Theta)$, the problem \eqref{eqdef} has at least one solution.
\end{theorem}

\begin{proof}
From Proposition \ref{pro1}, we have that $\|u^{\varepsilon_n}\|<c$, where $c$ is independent of $\varepsilon$ and $n$. Thus as $\varepsilon\rightarrow  0$, $n\rightarrow\infty$ and by considering subsequences if necessary, we may write that
\[
u^{\varepsilon_n}\rightarrow u,\quad\text{ weakly in }\mathcal V,
\]with $u\in\mathcal V$. Then, by the compactness of $\gamma$ the trace of $\mathcal V$ into $\mathrm L^2(\partial\mathscr O,\mathbb R^d)$, it follows that 
\[
\gamma\,u^{\varepsilon_n}\rightarrow \gamma\,u,\quad\text{ in }\mathrm L^2(\partial\mathscr O,\mathbb R^d).
\]
This implies that $u^{\varepsilon_n}_N=\gamma u^{\varepsilon_n}.\,n\rightarrow \gamma u.\,n=u_N$ in $L^2(\partial\mathscr O,\mathbb R^d)$ and thus $u^{\varepsilon_n}_N(z)\rightarrow u_N(z)$ a.e. $z\in\partial\mathscr O$.
Moreover due to Proposition \ref{pro2} we can write that
\begin{equation}
\Theta_\varepsilon(u^{\varepsilon_n}_N)\rightarrow  \xi,\quad\text{weakly in }  L^1(\partial{\mathscr O})
\end{equation}
By applying the properties of the Galerkin basis and a simple passage \`a la limite we obtain
\begin{equation}\label{eqlim}
\la \mathscr Au+\mathscr B[u],v\ra+\ds\int_{\partial\mathscr O} \xi v_N {\rm d}\sigma(x)=\la f,v\ra, \quad\forall v\in \widetilde{\mathcal V}
\end{equation}
from which it follows that 
\[
\left|\int_{\partial\mathscr O} \xi v_N {\rm d}\sigma(x)\right|\leq k\|v\|_{\mathcal V},
\]
and that the linear functional $\mathscr E$ can be uniquely extended to the whole space with $\mathscr E(\xi)\in \mathcal V^*$. Thus the expression \eqref{eqlim} can be written in the form

\begin{equation*}
\la \mathscr Au+\mathscr B[u],v\ra+\ds (\mathscr E(\xi),v)=\la f,v\ra,\quad\forall v\in \mathcal V.
\end{equation*}
In order to complete the proof it will be shown that
\begin{equation*}
\xi \in \hat{\Theta}(u_N(z))=\ds \partial j(u_N(z)),\quad\text{ for a.e }  z\in \partial{\mathscr O}.
\end{equation*}
As $u_N^{\varepsilon_n}\rightarrow u_N$ a.e., then by applying Egoroff's theorem we can find that for any $\alpha > 0$ we can determine $\Gamma\subset \partial{\mathscr O}$   with  $\sigma(\Gamma) <\alpha$  such that
\begin{equation*}
u_N^{\varepsilon_n}\rightarrow u_N,\quad\text{ uniformly on } \partial{\mathscr O} \setminus \Gamma,
\end{equation*} with $u_N\in \mathrm L^{\infty}(\partial{\mathscr O} \setminus \Gamma)$. Thus for any $\alpha > 0$ we can find $\Gamma \subset \partial{\mathscr O}$ with $\sigma(\Gamma) <\alpha$
such that for any $\mu> 0$ and for $\varepsilon < \varepsilon_0 < \mu/2$ and $n > n_0 > 2/\mu$ we have
\begin{equation*}
|u_N^{\varepsilon_n}- u_N|<\frac{\mu}{2},\quad\text{ on } \partial{\mathscr O}\setminus\Gamma.
\end{equation*}

Consequently, one obtain that

\begin{equation*}
\begin{array}{ll}
\Theta_\varepsilon(u^{\varepsilon_n}_N)&\leq \displaystyle\esssup_{|u_N^{\varepsilon_n}-  \xi|\leq \varepsilon} \Theta(\xi)\\
&\leq \displaystyle\esssup_{|u_N^{\varepsilon_n}-  \xi|\leq  \frac{\mu}{2}} \Theta(\xi)\\
&\leq \displaystyle\esssup_{|u_N-  \xi|\leq  \mu} \Theta(\xi)\\
&= \overline{\Theta}_\mu(u_N)
\end{array}
\end{equation*}
Analogously we prove the inequality

\begin{equation*}
 \underline{\Theta}_\mu(u_N)=\displaystyle\essinf_{|u_N -\xi|\leq \mu}\Theta(\xi)\leq \Theta_\varepsilon(u_N^{\varepsilon_n})
\end{equation*}

We take now $v \geq 0$ a.e. on $\partial{\mathscr O}\setminus\Gamma$  with $v\in \mathrm L^\infty(\partial{\mathscr O}\setminus\Gamma ) $. This implies 

\begin{equation*}
\int_{ \partial{\mathscr O}\setminus\Gamma} \underline{\Theta}_\mu(u_N)\,v {\rm d}\sigma\leq \int_{ \partial{\mathscr O}\setminus\Gamma} \Theta_\varepsilon(u_N^{\varepsilon_n} )\,v {\rm d}\sigma\leq 
\int_{ \partial{\mathscr O}\setminus\Gamma}\overline{\Theta}_\mu(u_N)\,v {\rm d}\sigma
\end{equation*}

Taking the limit $\varepsilon \rightarrow 0$ as $n\rightarrow \infty$ we obtain that
\begin{equation*}
\int_{ \partial{\mathscr O}\setminus\Gamma} \underline{\Theta}_\mu(u_N)\,v {\rm d}\sigma\leq \int_{ \partial{\mathscr O}\setminus\Gamma} \xi\, v {\rm d}\sigma\leq\int_{ \partial{\mathscr O}\setminus\Gamma}\overline{\Theta}_\mu(u_N)\,v {\rm d}\sigma
\end{equation*}
and as $\mu\rightarrow 0$  that
\begin{equation*}
\int_{ \partial{\mathscr O}\setminus\Gamma} \underline{\Theta}(u_N)\, v {\rm d}\sigma\leq \int_{ \partial{\mathscr O}\setminus\Gamma} \xi\, v {\rm d}\sigma\leq\int_{ \partial{\mathscr O}\setminus \Gamma}\overline{\Theta}(u_N)\,v {\rm d}\sigma
\end{equation*}

Since $v$ is arbitrary we have that

\begin{equation*}
\xi\in [ \underline{\Theta}(u_N), \overline{\Theta}(u_N)     ]=  \widehat{\Theta}(u_N)
\end{equation*}

where $\sigma(\partial\mathscr O) < \alpha $. For  $\alpha $ as small as possible, we obtain the result.

\end{proof}
Several of the arguments applied in the proof of this theorem are borrowed
from the method developed in \cite{rauch1977discontinuous} for the existence proof for semilinear
differential equations.

\begin{remark}\label{rem}
The question of uniqueness is more delicate. In fact even if we suppose the following monotonicity type assumption on $\Theta$ in the way did in \cite{M05HE}
\[
 \displaystyle\essinf_{\xi_1\neq\xi_2}\frac{\Theta(\xi_1)-\Theta(\xi_2)}{\xi_1-\xi_2}
 >-m
 \]a problem occurs when one needs to get estimates in $L^2(\partial\mathscr O)$ for $\xi\in L^1(\partial\mathscr O)$, where $(u,\xi)$ is a solution of \eqref{eqdef}. But as $L^2(\partial\mathscr O)\subset L^1(\partial\mathscr O)$ it may occur that $\xi\in L^1(\partial\mathscr O)\setminus L^2(\partial\mathscr O)$. As we have a weak assumption and we don't make use of any type of growth conditions we have in fact enlarged the space where we are looking for a solution. By doing so we loose any hope to prove a uniqueness result without a growth condition.

\end{remark}

\section{Dependence result}

In this section, we characterize the dependence of solutions on the hemivariational part, particularly on the functions $\Theta$. Consider a sequence of functions $\Theta^k$ converging in some sense to a function $\Theta^\infty$. Our aim is to prove that the solutions $u^k$ constructed from $\Theta^k$ converge to $u^\infty$ corresponding to the function $\Theta^\infty$. To do so we make the following hypothesis:
\\
\begin{enumerate}
\item[ $(H^k)$ ] 
 \begin{enumerate}
 \item[(i) ] $(\Theta^k)_{k\in\mathbb N}\subset L^\infty_{loc}(\real)$ and $\Theta^k(t \pm 0)$ exists for any  $t \in \real$ and $k\in\mathbb N$.
 \item[ (ii) ] there is  $t_{0}>0$ such that
\begin{equation}
 \sup_{ ]-\infty,-t_{0}[ } \Theta^k(t) \leq 0 \leq \inf _{ ]t_{0},+\infty[ } \Theta^k(t),\quad \text{ for all }k\in\mathbb N
 \end{equation}
 \end{enumerate}
\item[$ (H^\infty ) $ ]

\begin{enumerate}
\item[ (iii) ] $\Theta^\infty\in L^\infty_{loc}(\real)$ and  $\Theta^\infty(t \pm 0)$ exists for any  $t \in \real$.
\item[ (iv) ] $\limsup_{k\to \infty}\mathrm{Graph}(\widehat{\Theta}^k(.))\subset \mathrm{Graph}(\widehat{\Theta}^\infty(.)) $ (in the sense of Kuratowski, see \cite{AF90}).
\end{enumerate}
\end{enumerate}
\begin{theorem}
Assume $(H^k)$ and $(H^\infty)$ hold and $f\in\mathcal V^\star$. Let $(u^k)_{k\in\mathbb N}$ denotes a sequence of solutions of the problem \eqref{eqdef}, where $\Theta$ is replaced by $\Theta^k$. Then there exists a subsequence of $(u^k)_{k\in\mathbb N}$(denoted by the same symbol) such that $u^k\rightarrow u^\infty$ weakly in $\mathcal V$, where $u^\infty\in \mathcal V$ is a solution to \eqref{eqdef} corresponding to $\Theta^\infty$.
\end{theorem}

The assumption $(H^k)(ii)$ is slightly stronger than the one needed usually to ensure the existence of solutions but not too restrictive. As we will see in the proof of the theorem one can always find a constant $\delta>0$ such that $\Theta^k$ is positive in $]\delta,+\infty[$ including for the discontinuities(which is not the case for the usual assumption). This Change allow us to have similar lower bound for the integral part, i.e. $\ds\int_{\partial{\mathscr O}} \Theta^k(u_N^k(z)\,u_N^k(z)\,{\rm d}\sigma(z)$ which make it possible to obtain the boundedness of $(u^k)_k$ and the weak precompactness of $\{\Theta^k(u_N^k)\}$ in $\mathrm L^1(\partial{\mathscr O})$.
\begin{proof}

The sequence $(u^k)_{k\in\mathbb N}$ is bounded. In fact, from $(H^k)$ (ii) we can find $\delta_1>0$ and $\delta_2>0$ such that 

\begin{equation*}
\left\{
\begin{array}{ll}
\Theta^k(t)\geq 0  &\text{ si } t>\delta_1\\
\Theta^k(t)\leq 0  &\text{ si } t<-\delta_1\\
|\Theta^k(t)|\leq \delta_2  &\text{ si } |t|\leq\delta_1\\
\end{array}
\right.
\end{equation*}

This in hand, one can prove in the same way as Lemma \ref{lem1} that
\begin{equation}\label{delta}
\int_{\partial{\mathscr O}} \Theta^k(u_N^k(z)).\,u_N^k(z)\,{\rm d}\sigma(z)\geq -\delta_1\delta_2\,\sigma(\partial{\mathscr O})
\end{equation}
Now by using the fact that $u^k$ is a solution of \eqref{eqdef} by replacing $\Theta$ by $\Theta^k$, we have
\[
\langle\mathscr A u^k,u^k\rangle+\langle\mathscr \mathscr B[u^k],u^k\rangle+\int_{\partial{\mathscr O}} \Theta^k(\gamma\,u_N^k(z))\gamma\,u_N^k(z)\,{\rm d}\sigma(z)=\la f,u^k\ra 
\]which leads, from \eqref{delta} and the coerciveness of $\mathscr A$, to $\alpha\|u^k\|^2-\delta_1\delta_2\sigma(\partial{\mathscr O})\leq c\|u^k\|$. If $\|u^k\|$ were unbounded(i.e. $\|u^k\|\rightarrow+\infty$ as $k\rightarrow +\infty$) it will leads to contradiction.
Now, by considering a subsequence if necessary, we may write
\[
u^k\rightarrow u^\infty,\quad\text{ weakly in }\mathcal V,
\]with $u^\infty\in\mathcal V$. Then, by the compactness of $\gamma$ the trace of $\mathcal V$ into $\mathrm L^2(\partial\mathscr O,\mathbb R^d)$, it follows that 
\[
\gamma\,u^k\rightarrow \gamma\,u^\infty,\quad\text{ in }\mathrm L^2(\partial\mathscr O,\mathbb R^d).
\]
This implies that $u_N^k=u^k.\,n\rightarrow u^\infty.\,n=u_N^\infty$ in $\mathrm L^2(\partial\mathscr O,\mathbb R^d)$ and thus $u_N^k(z)\rightarrow u_N^\infty(z)$ a.e. $z\in\partial\mathscr O$.

By Dunford-Pettis theorem, one can get, without much difficulties, the weak compactness in $\mathrm L^1(\partial{\mathscr O})$ of $\{\Theta^k(u_N^k)\}$. It follows that there exists $\Xi^\infty\in \mathrm L^1(\partial{\mathscr O})$ such that 
\[
\Theta^k(u_N^k)\rightarrow \Xi^\infty \text{ as } k\to+\infty\quad \text{weakly in } \mathrm L^1(\partial{\mathscr O}).
\]
Moreover, as $k\to+\infty$, we have
\[
\la\mathscr A\,u^k+\mathscr B[u^k],v\ra\rightarrow\la\mathscr A\,u^\infty+\mathscr B[u^\infty],v\ra,\quad\text{ for all }v\in\tilde{\mathcal V}.
\]
It follows that 
\[
\langle\mathscr A u^\infty+ \mathscr B[u^\infty],v\rangle+\int_{\partial{\mathscr O}} \Xi.\, v_N(z)\,{\rm d}\sigma(z)=\la f,v\ra,\quad \text{ for all }\tilde v\in\tilde{\mathcal V}.
\]
On the other hand $\tilde{\mathcal V}$ is dense in $\mathcal V$, and from the equality

\[
\int_{\partial{\mathscr O}} \Xi.\, v_N(z)\,{\rm d}\sigma(z)=\la f,v\ra-\langle\mathscr A u^\infty+ \mathscr B[u^\infty],v\rangle,\quad \text{ for all }\tilde v\in\tilde{\mathcal V},
\]it follows that 
\[
\mathscr A u^\infty+ \mathscr B[u^\infty]+ \mathscr E(\Xi)= f,\quad \text{ weakly in  }\mathcal V^*
\]
It remains to show that $\Xi\in\widehat{\Theta}^\infty(u^\infty_N)$ a.e. in $\partial\mathscr O$.
Recall that $u_N^k(z)\rightarrow u_N^\infty(z)$ a.e. $z\in\partial\mathscr O$. By applying Egoroff's theorem we find that for any $\alpha>0$ we can determine $\Gamma\in\partial\mathscr O$ with $\sigma(\Gamma)<\alpha$ such that 
\[
u_N^k\rightarrow u_N^\infty, \quad\text{uniformly on }\partial\mathscr O\setminus\Gamma.
\]
By applying Theorem 7.2.1 of Aubin and Frankowska \cite{AF90} we deduce
\[
\Xi\in\overline{\mathrm{conv}}\left(\limsup_{t\to u_N^\infty(z),\,k\to\infty}\widehat{\Theta}^k(t)\right)\subset\widehat{\Theta}^\infty(u_N^\infty(z)),\quad\text{for all }z\in \partial\mathscr O\setminus\Gamma.
\]
The latter inclusion follows from the assumption $(H^\infty)(iv)$, where $\sigma(\Gamma)<\alpha$. For $\alpha$ as small as possible, we obtain the result.

\end{proof}

\section{Optimal Control}
In this section, we provide a result on dependence of solutions with respect to the density of the external forces and use it to study the distributed parameter optimal control problem corresponding to it.

Before we start to discuss the optimal control problem, we first prove the following auxiliary result.
\begin{theorem}\label{fn}
Under $H(\Theta)$ assume that $f_n$, $f\in L^2(\mathscr O;\mathbb R^d)$ such that $f_n\rightarrow f$ weakly in $L^2(\mathscr O;\mathbb R^d)$. Then for every $\{u^n\}_n$ solution to the problem \eqref{eqdef} corresponding to $f_n$, we can find a subsequence (still denoted with the same symbol) such that $u^n\rightarrow u$ in $\mathcal V$ and $u$ is a solution to problem \eqref{eqdef} corresponding to $f$.

\end{theorem}
\begin{proof}
Let $f_n,\,f\in  L^2(\mathscr O;\mathbb R^d)$ with $f_n\rightarrow f$ weakly in $L^2(\mathscr O;\mathbb R^d)$. Then by Theorem \ref{existence}, there exists $(u^n,\xi_n)\in\mathcal V\times L^1(\partial\mathscr O;\mathbb R)$ such that $\xi_n\in\partial j(u_N^n)$ and 
\[
\langle\mathscr A u^n+ \mathscr B[u^n],v\rangle+\int_{\partial{\mathscr O}} \xi_n\, v_N(z)\,d\sigma(z)=\la f_n,v\ra,\quad \text{ for all } v\in\widetilde{\mathcal V}
\]
It is possible to prove a similar result as Lemma \ref{lem1} with eventually different constants. By the coerciveness of $\mathscr A$ and the continuity of the injection $L^2(\mathscr O;\mathbb R^d)\subset \mathcal V^*$, we get 
\[
M\,\|u^n\|_{\mathcal V}^2-ab\sigma(\partial\mathscr O)-\|f_n\|_{L^2(\mathscr O;\mathbb R^d)}\,\|u^n\|_{\mathcal V}\leq 0,
\]
which simplifies to
\[
\|u^n\|_{\mathcal V}\leq \frac{ab\sigma(\partial\mathscr O)}{\|u^n\|_{\mathcal V}}+\|f_n\|_{L^2(\mathcal V)}.
\]
One can see immediately that if $\|u^n\|_{\mathcal V}$ converges to $+\infty$, so will do $\|f_n\|_{L^2(\mathscr O;\mathbb R^d)}$, which means that $\{u^n\}_n$ is in fact bounded(with bound independent of $n$).  As $\mathcal V$ is a reflexive Banach space, we may assume, by passing to a subsequence if necessary, that there exists $u\in\mathcal V$ such that $u^n$ converges to $u$. From the continuity of $\mathscr A$ and $\mathscr B[.]$, we have $\mathscr A\,u^n\rightarrow \mathscr A\,u$ and $\mathscr B[u^n]\rightarrow\mathscr B[u]$ weakly in $\mathcal V^*$. Using the compactness of the trace operator $\gamma$, we may assume that $\gamma\,u^n\rightarrow \gamma\, u$ in $L^2(\mathscr O;\mathbb R^d)$ and then $\gamma\,u^n(z)\rightarrow \gamma\, u(z)$ for a.e. $z\in\partial\mathscr O$. Consequently, $u_N^n(z)\rightarrow u_N(z)$ for a.e. $z\in\partial\mathscr O$. On the other hand one have
\[
\int_{\partial{\mathscr O}} \xi_n\, v_N(z)\,{\rm d}\sigma(z)=\la f_n-\mathscr A u^n-\mathscr B[u^n],v\ra,\quad \text{ for all } v\in\widetilde{\mathcal V}
\]which means that $\xi_n\rightarrow \xi$ in $L^1(\partial\mathscr O)$ and 
\[
\langle\mathscr A u+ \mathscr B[u],v\rangle+\int_{\partial{\mathscr O}} \xi\, v_N(z)\,d\sigma(z)=\la f,v\ra,\quad \text{ for all } v\in\widetilde{\mathcal V},
\]
in order to complete the proof it will be shown that 
\[
\xi\in\partial j(u_N(z)) \text{ for a.e } z\in \partial\mathscr O.
\] 
To do this, we first show that  
\[
\partial j(u_N^n)\subset \partial j(u_N),\quad \text{for all }n .
\]
As $u_N^n\rightarrow u_N$ for a.e. $z\in\partial\mathscr O$, we can find, by Egoroff's theorem, that for any $\alpha > 0$ we can determine $\Gamma\subset \partial{\mathscr O}$   with  $\sigma(\Gamma) <\alpha$  such that
\begin{equation}
u_N^{n}\rightarrow u_N,\quad\text{ uniformly on } \partial{\mathscr O} \setminus \Gamma,
\end{equation} with $u_N\in \mathrm L^{\infty}(\partial{\mathscr O} \setminus \Gamma)$. Thus for any $\mu> 0$ there exists $n_0$ such that for all $n > n_0$ we have
\begin{equation*}
|u_N^{n}(z)- u_N(z)|<\frac{\mu}{2},\quad\forall z\in \partial{\mathscr O}\setminus\Gamma.
\end{equation*}

By using triangle inequality, we have that 

\begin{equation*}
\begin{array}{ll}
\overline{\Theta}_{\frac{\mu}{2}}(u^{_n}_N)&= \displaystyle\esssup_{|u_N^{n}-  \xi|\leq \frac{\mu}{2}} \Theta(\xi)\\
&\leq \displaystyle\esssup_{|u_N-  \xi|\leq  \mu} \Theta(\xi)\\
&= \overline{\Theta}_{\mu}(u_N).
\end{array}
\end{equation*}
Analogously we prove the inequality

\begin{equation*}
 \underline{\Theta}_{\mu}(u_N)\leq \underline{\Theta}_{\frac{\mu}{2}}(u^n_N).
\end{equation*}
Taking the limit as $\mu\rightarrow 0^+$, we obtain that 
\[
\xi_n\in\widehat{\Theta}(u_N^n(z))\subset \widehat{\Theta}(u_N(z)), \quad \forall n\geq n_0,\, z\in\partial \mathscr O\setminus\Gamma,
\]
from which we conclude that 
\[
\xi(z)\in \overline{\mathrm{conv}}\,\, \widehat{\Theta}(u_N(z))=\widehat{\Theta}(u_N(z)),\quad\forall\,z\in \partial \mathscr O\setminus\Gamma,
\]where $\sigma(\Gamma) < \alpha $. For  $\alpha $ as small as possible, we obtain the result.

\end{proof}

We follow Mig\'{o}rski \cite{law2013note} and we let $\mathcal U=L^2(\mathscr O;\mathbb R^d)$ be the space of controls. For every $f\in\mathcal U$, we denote by $S(f)\subset\mathcal V\times L^1(\partial\mathscr O;\mathbb R)$ the solution set corresponding to $f$ of the problem \eqref{eqdef}. It is then clear that, by definition, $S(f)$ is nonempty for all $f\in\mathcal U$.

Let $\mathcal U_{ad}$ be a nonempty subset of $\mathcal U$ consisting of admissible controls. Let $\mathscr F:\mathcal U\times\mathcal V\times  L^1(\partial\mathscr O;\mathbb R)\rightarrow \mathbb R$ be the objective functional we want to minimize. The control problem reads as follows: Find a control $\hat f\in\mathcal U_{ad}$ and a state $(\hat u,\hat\xi)\in S(\hat f)$ such that
\begin{equation}\label{control}
\mathscr F(\hat f,\hat u,\hat\xi)=\inf\left\{\mathscr F(f,u,\xi):\,f\in\mathcal U_{ad},\, (u,\xi)\in S( f)  \right\}.
\end{equation}
A triple which solves \eqref{control} is called an optimal solution. The existence of such optimal control can be proved by using Theorem \ref{fn}. To do so, we need the following additional hypothesis:

\begin{itemize}
\item[$H(\mathcal U_{ad})$]$\quad \mathcal U_{ad}$ is a bounded and weakly closed subset of $\mathcal U$.
\end{itemize}

\begin{itemize}
\item[$H(\mathscr F)$] $\quad \mathscr F$ is lower semicontinuous with respect to $\mathcal U\times\mathcal V\times  L^1(\partial\mathscr O;\mathbb R)\rightarrow \mathbb R$ endowed with the weak topology.
\end{itemize}

\begin{theorem}
Assume that $H(\Theta)$, $H(\mathcal U_{ad})$ and $H(\mathscr F)$ are fulfilled. Then the problem \ref{control} has an optimal control.
\end{theorem}

\begin{proof}
Let $(f_n,u_n,\xi_n)$ be a minimizing sequence for the problem \eqref{control}, i.e $f_n\in\mathcal U_{ad}$ and $(u_n,\xi_n)\in S(f_n)$ such that 
\[
\displaystyle\lim_{n\to\infty}\mathscr F(f_n,u_n,\xi_n)=\inf\left\{\mathscr F(f,u,\xi):\,f\in\mathcal U_{ad},\, (u,\xi)\in S( f) \right\}=:m
\]
It follows that the sequence $f_n$ belongs to a bounded subset of the reflexive Banach space $\mathcal V$. We may then assume that $f_n\rightarrow \hat f$ weakly in $\mathcal V$ (by passing to a subsequence if necessary). By $H(\mathcal U_{ad})$, we have $\hat f\in\mathcal U_{ad}$. From Theorem \ref{fn}, we obtain, by again passing to a subsequence if necessary, that $u_n\rightarrow \hat u$ weakly in $\mathcal V$ with $(\hat u,\hat\xi)\in S(\hat f)$. By $H(\mathscr F)$, we have $m\leq\mathscr F(\hat f,\hat u,\hat\xi)\leq \displaystyle\liminf_{n\to\infty}\mathscr F(f_n,u_n,\xi_n)=m$. Which completes the proof.
\end{proof}


\begin{acknowledgements}
We thank Prof. S. Mig\'{o}rski for pointing out that the Rauch and the growth conditions  are completely independent.
\end{acknowledgements}
 \section*{Conflict of interest}
The authors declare that they have no conflict of interest.

\end{document}